\documentclass[11pt,reqno]{amsart}
\usepackage{amsfonts}
\usepackage{amssymb,color}
\usepackage{amssymb}

\usepackage[margin=2.8cm, a4paper]{geometry}

\def\normo#1{\left\|#1\right\|}
\def\normb#1{\big\|#1\big\|}

\def\aabs#1{\left|#1\right|}
\def\brk#1{\left(#1\right)}
\def\bbrk#1{\big (#1\big )}

\def\norm#1{\|#1\|}
\def\jb#1{\langle#1\rangle}

\newcommand{\R}{{\mathbb R}}

\newcommand{\Z}{{\mathbb Z}}
\newcommand{\ft}{{\mathcal{F}}}
\newcommand{\Hl}{{\mathcal{H}}}

\newcommand{\les}{{\lesssim}}
\newcommand{\ges}{{\gtrsim}}

\newcommand{\supp}{{\mbox{supp}}}

\def\jb#1{\langle#1\rangle}
\def\norm#1{\|#1\|}
\def\normo#1{\left\|#1\right\|}
\def\normb#1{\big\|#1\big\|}

\def\aabs#1{\left|#1\right|}

\newcommand{\EQ}[1]{\begin{equation}\begin{split} #1 \end{split}\end{equation}}
\newcommand{\Del}[1]{}
\newcommand{\CAS}[1]{\begin{cases} #1 \end{cases}}

\numberwithin{equation}{section}

\newtheorem{thm}{Theorem}[section]

\newtheorem{lem}[thm]{Lemma}
\newtheorem{prop}[thm]{Proposition}

\theoremstyle{remark}

\newtheorem{theorem}{Theorem}[section]

\theoremstyle{remark}
  \newtheorem{remark}[theorem]{Remark}

\theoremstyle{definition}

\begin{document}

\title[Decay estimates for NLS]{Pointwise Decay of solutions to the energy critical nonlinear Schr\"odinger equations}\thanks{Z. Guo is supported by ARC DP200101065. C. Huang is supported by NNSF of China (No. ~11971503). L. Song is supported by  NNSF of China (No.~12071490)}
\author[Z. Guo]{Zihua Guo}
\address{School of Mathematical Sciences, Monash University, VIC 3800, Australia }
\email{zihua.guo@monash.edu}

\author[C. Huang]{Chunyan Huang}
\address{School of Statistics and Mathematics, Central University of Finance and Economics, Beijing
100081, China}
\email{hcy@cufe.edu.cn}

\author[L. Song]{Liang Song}
\address{School of Mathematics, Sun Yat-sen Univeristy, Guangzhou 510275, China}
\email{songl@mail.sysu.edu.cn}

\begin{abstract}
In this note, we prove pointwise decay in time of solutions to  the 3D energy-critical nonlinear Schr\"odinger equations assuming data in $L^1\cap H^3$.  The main ingredients are the boundness of the Schr\"odinger propagators in Hardy space due to Miyachi \cite{Miyachi} and a fractional Leibniz rule in the Hardy space. We also extend the  fractional chain rule to the Hardy space.
\end{abstract}

\maketitle

\section{Introduction}

Consider the nonlinear Schr\"odinger equation
\EQ{\label{eq:NLS}
\CAS{i\partial_t u+\Delta u= \mu |u|^{4}u, \quad (x,t)\in \mathbb{R}^3\times \R,\\
u(x,0)=u_0(x)}}
with $\mu\in \{1,-1\}$.
It is well-known that the solutions of the linear Schr\"odinger equation satisfy the following dispersive estimates
\begin{equation}\label{decay-NLS}
\|e^{it\Delta}u_0\|_{L^{\infty}_x}\leq C |t|^{-3/2}\|u_0\|_{L^1_x}.
\end{equation}
It is natural to ask whether one can obtain global solutions $u$ to the nonlinear Schr\"odinger equation \eqref{eq:NLS} with the same time decay, namely
\begin{equation}\label{decay-NLS2}
\|u(t)\|_{L^{\infty}_x}\leq C |t|^{-3/2}.
\end{equation}

The equation \eqref{eq:NLS} is energy-critical. More precisely, \eqref{eq:NLS} is invariant under the following scaling transform
\EQ{
u(t,x)\to \lambda^{1/2} u(\lambda x, \lambda^2 t), \quad u_0(x)\to \lambda^{1/2} u_0(\lambda x)
}
and $\norm{\lambda^{1/2} u_0(\lambda x)}_{\dot H^1}=\norm{u_0}_{\dot H^1}$.
Global well-posedness and scattering theory for \eqref{eq:NLS} were extensively studied in the last two decades (e.g. see \cite{I-team, KM} for the introduction).  A global solution $u$ of \eqref{eq:NLS} scatters in the energy space means there exists $\phi_\pm\in H^1$ such that
\EQ{
\lim_{t\to \pm \infty} \norm{u(t)-e^{it\Delta}\phi_\pm}_{H^1}=0.
}
Even if we have scattering in $H^1$, to get pointwise-in-time decay for solutions of \eqref{eq:NLS} is not trivial. On one hand, one does not expect pointwise-in-time decay if assuming initial data only in $H^1(\R^3)$ in view of the linear Schr\"odinger equation. On the other hand, assuming $u_0\in L^1$ one cannot ensure the scattering state $\phi_{\pm}\in L^1$.  Thus to obtain pointwise decay requires some extra effort.  Recently, by contradiction argument Fan and Zhao \cite{FZ} proved that scattering solution of \eqref{eq:NLS} satisfies \eqref{decay-NLS2} assuming $u_0\in L^1\cap H^k$ for some $k$.

In this note, we consider the finer asymptotic behaviour based on scattering results in the energy space, and hence give a direct and simpler proof of Fan-Zhao's result for \eqref{eq:NLS}.  Moreover, our result is more quantitative.  We can show that the scattering state $\phi_\pm\in L^1$ and reversely the wave operator is also defined in $L^1$.  
The main result is
\begin{thm}\label{Th1}

(1) Let $u$ be a global solution of \eqref{eq:NLS} with finite scattering norm
\EQ{
\norm{u}_{L^{10}_{t,x}(\R\times \R^3)}\leq K
}
for some $K>0$. Assume in addition $u(0)\in H^{3}(\R^3)\cap L^1(\R^3)$. Then  $u\in C(\R: H^3)$,  $e^{-it\Delta}u(t)\in C(\R;L^1)$ and
\EQ{
\|u(t,x)\|_{L^{\infty}_x}\leq C(K, u_0) |t|^{-3/2}.
}
Moreover, there exists $\phi_\pm \in H^{3}(\R^3)\cap L^1(\R^3)$ such that
\EQ{
\lim_{t\to \pm \infty} (\norm{u(t)-e^{it\Delta}\phi_\pm}_{H^3}+&\norm{e^{-it\Delta }u(t)-\phi_\pm}_{L^1})=0\\
 \norm{u(t)-e^{it\Delta}\phi_\pm}_{L^\infty}\leq& C|t|^{-5}, \quad \pm t>0.
}

(2) Assume $\phi_+ \in H^3(\R^3)\cap L^1(\R^3)$.  Then there exists a unique solution $u\in L_{t,x}^{10}$ to \eqref{eq:NLS} such that $u\in C(\R: H^3)$, $e^{-it\Delta}u(t)\in C(\R;L^1)$ and 
\EQ{
\lim_{t\to  \infty} ( \norm{u(t)-e^{it\Delta}\phi_+}_{H^3}+& \norm{e^{-it\Delta }u(t)-\phi_+}_{L^1})=0, \\
\norm{u(t)-e^{it\Delta}\phi_+}_{L^\infty}\leq &C|t|^{-5}, \quad t>0.
}
Similar results hold for $t\to -\infty$.
\end{thm}

It is known that global well-posedness and scattering of \eqref{eq:NLS} holds in
$\dot{H}^1(\R^3)$ and
\EQ{
\norm{u}_{L^{10}_{t,x}(\R\times \R^3)}\leq C_{\norm{u_0}_{\dot{H}^1}}
}
provided that either of the following conditions hold:
\begin{itemize}
    \item in the defocusing case $\mu=1$ (see \cite{I-team});
    \item in the focusing case $\mu=-1$ (see \cite{KM}): $E(u_0)<E(W)$, $\norm{\nabla u_0}_{L^2}<\norm{\nabla W}_{L^2}$, $u_0$ radial, where 
    \EQ{
    E(u)=\int \frac{1}{2}|\nabla u|^2+\frac{1}{6}|u|^6dx
    }
    and $W=(1+|x|^2/3)^{-1/2}$.
\end{itemize}
Morever, assuming additional conditions $u_0\in H^3(\R^3)$, one can obtain the persistence (scattering in $H^3$) and
\EQ{\label{eq:spacetimebound}
\norm{\jb{\nabla}^3 u}_{L_t^q L_x^r(\R\times \R^3)}\leq C_{\norm{u_0}_{H^3}}
}
where $(q,r)$ satisfies the admissible conditions: $2\leq q,r\leq \infty$ and $\frac{1}{q}=\frac{3}{2}(\frac{1}{2}-\frac{1}{r})$.

It is known (see \cite{HN}) that for nice initial data $f$ we have for $\gamma>\frac{n}{2}+2\beta$
\begin{equation}
e^{it\Delta}f(x)=\frac{1}{(i2t)^{n/2}}e^{i\frac{|x|^2}{4t}}\hat f(x/t)+\frac{1}{t^{n/2+\beta}}O(\norm{\jb{x}^\gamma f}_{L^2}), \quad t\to \pm\infty.
\end{equation}
Scattering and the wave operator in $L^1$ thus provide a precise asymptotic profile of the nonlinear solutions.

Our main ingredient of the proof is the boundedness for the Schr\"odinger propagator in Hardy space which was proved by Miyachi \cite{Miyachi} and a fractional Leibniz rule in Hardy space.

\section{Proof of the main Theorem}

Throughout this note, we use $C$ to denote some universal constant which may change from line to line. 
For $X, Y \in \mathbb{R}$, $X\lesssim Y$
means $X\leq CY$ for some $C>0$, similarly for $X\ges Y$.  $X\sim Y$ means $X\les Y$ and $X\ges Y$.

Fix a bump function $\eta\in C_0^\infty(\R^d)$ such that $\eta$ is non-negative, radial and radially decreasing, $\supp \eta\subset \{|x|\leq 1.1\}$ and $\eta(x)\equiv 1$ for $|x|\leq 1$.
Let $\chi(x)=\eta(x)-\eta(2x)$. 
For $k\in \Z$, define the operators
\EQ{
\dot P_k=\ft^{-1} \chi(\frac{x}{2^k})\ft, \quad P_{\leq k}=\ft^{-1} \eta(\frac{x}{2^k})\ft, \quad P_{>k}=I-P_{\leq k}
}
and
\EQ{
P_k=
\CAS{
\ft^{-1} \chi(\frac{x}{2^k})\ft, & \quad k> 0;\\
\ft^{-1} \eta(x)\ft, &\quad k=0;\\
0, &\quad k\leq -1.
}
}
Here $\ft$ is the Fourier transform: $\ft f(\xi)=\int_{\R^d} e^{-ix\xi}f(x)dx$. We define $D^s=\ft^{-1}|\xi|^s\ft$ and $J^s=\ft^{-1}(1+|\xi|^2)^{s/2}\ft$.

Besov space $B^s_{p,q}$ is the Banach space with the norm
\EQ{
\norm{f}_{B^s_{p,q}}=\normb{2^{ks}\norm{P_k f}_{L^p}}_{l_k^q}
}
and Triebel-Lizorkin space $F^s_{p,q}$ is the Banach space with the norm
\EQ{
\norm{f}_{F^s_{p,q}}=\normb{\norm{2^{ks} P_k f}_{l_k^q}}_{L^p}.
}
Homegeneous versions $\dot{B}^s_{p,q}$ and $\dot{F}^s_{p,q}$ are defined similarly by replacing $P_k$ with $\dot P_k$.
It is well-known that 
\EQ{
\dot{F}^{0}_{1,2} \sim \Hl^1, \quad {F}^{0}_{1,2} \sim h^1,
}
where $\Hl^1$ is the Hardy space and $h^1$ is the local Hardy space. $h^1$ is embedded into $L^1$.

Let $M$ be the Hardy-Littlewood maximal operator
\EQ{
Mf(x)=\sup_{r>0}\frac{1}{|B_r(x)|}\int_{B_r(x)} |f(y)|dy.
}
The maximal operator can control many $L^1$-average type operators.  
We recall the well-known pointwise maximal function estimate, see~\cite{Stein93}.
\begin{lem} \label{radialmaxicon}
Let $g(x)$ be a nonnegative radial decreasing integrable function, suppose $|\psi(x) | \leq g(x)$ almost everywhere and $f \in L_{loc}^1(\R^d)$,  then 
\begin{align*}
|\psi_{\epsilon} * f(x)| \leq C \|g\|_{L^1}\cdot M(f)(x), \ \ \ \ \forall\, \epsilon>0,
\end{align*} 
where $\psi_{\epsilon}(x) = \epsilon^{-d} \psi(\epsilon^{-1}x)$.
\end{lem}

We will also need the following result concerning the boundedness of $M$ acting on vector-valued functions (see \cite{Stein93}).
\begin{lem}[Maximal inequality]\label{lem:maxi}
Let $(p,q)\in (1,\infty)\times (1,\infty] $ or $p=q=\infty$ be given. Suppose $ \{f_j\}_{j\in \mathbb{Z}}$ is a sequence of functions in $L^p(\R^d)$ satisfying $\|f_j\|_{l_j^q(\mathbb{Z})} \in L^p(\R^d)$,  then 
\EQ{
\| M(f_j)(x)\|_{L_x^pl_j^q} \leq C \|f_j(x)\|_{L_x^pl_j^q}
}
for some constant $C=C(p,q)$.
\end{lem}

It is worth noting that the above Lemma fails for $p=1$, which causes some difficulty when dealing with $L^1$-based space (e.g. Hardy space).  This was usually overcome by the following variant maximal estimate.

\begin{lem}[\cite{Triebel}]\label{lem: scalarmaxiesti}
Assume $0<r<\infty$.  There exists $C>0$ such that for all $R>0$
\begin{align}
\sup_{y\in \R^d} \frac{|f(x-y)|}{(1+|R y|)^{\frac{d}{r}}} \leq C \big[M(|f|^r)(x)\big]^{\frac{1}{r}}, \quad \forall x\in \R^d
\end{align} 
holds for all $f$ with $\supp \hat{f}\subset B(R)$.
\end{lem}

The above lemma is useful for linear estimate.  Motivated by \cite{GL}, we derive the following improvement which is useful for the multilinear estimates. 

\begin{lem} \label{lem:keyesti}
Let $L>0$, $j,k\in \mathbb{Z}$, $j>k-L$ and $r\in (0,1]$. Assume $\psi_k \in L^1(\R^d)$  satisfies for some $A>0$ and any $f\in L_{loc}^1(\R^d)$
\begin{align} \label{decayhy}
\aabs{[|\psi_k(y)|(1+|2^k y|)^{(\frac{1}{r}-1)d}]*f(x)}\leq A \cdot M(f)(x),\quad \forall x\in \R^d.
\end{align}
Then there exists a constant $C=C(A,L)$ such that
\begin{align} \label{convomaxi}
|(\psi_k * f)(x)| \leq C 2^{(j-k){(\frac{1}{r}-1)d} }  [ M(|f|^r)(x)]^{\frac{1}{r}}
\end{align}
holds for all $f$ with $\supp \hat{f}\subset B(c2^j)$.
\end{lem}
\begin{proof}
We have 
\EQ{
&|(\psi_k*f)(x)| \\
\leq& \int_{\R^d} |\psi_k(y)|\cdot |f(x-y) | dy  \\
\leq& \int_{\R^d} 2^{-kd}  |\psi_k(2^{-k}y)|\cdot |f(x-2^{-k}y) | dy  \\
\leq& \int_{\R^d} 2^{-kd}  |\psi_k(2^{-k}y)|\cdot |f(x-2^{-k}y)|^{r} (1+2^{j-k}|y|)^{(\frac{1}{r}-1)d}
\cdot \sup_{y\in \R^d} \frac{|f(x-2^{-k}y)|^{1-r}}{(1+2^{j-k}|y|)^{(\frac{1}{r}-1)d}} dy.
}
In view of Lemma \ref{lem: scalarmaxiesti}, one can see
\[
\sup_{y\in \R^d} \frac{|f(x-2^{-k}y)|^{1-r}}{(1+2^{j-k}|y|)^{(\frac{1}{r}-1)d}}
\leq C[ M(|f|^r)(x)]^{\frac{1-r}{r}}.
\]
Then we get
\EQ{
&|(\psi_k*f)(x)| \\
\les& \int_{\R^d} 2^{-kd}  |\psi_k(2^{-k}y)|\cdot |f(x-2^{-k}y)|^{r} (1+2^{j-k}|y|)^{(\frac{1}{r}-1)d}dy\cdot [ M(|f|^r)(x)]^{\frac{1-r}{r}}\\
\les& 2^{(j-k)(\frac{1}{r}-1)d}\int_{\R^d} 2^{-kd}  |\psi_k(2^{-k}y)|\cdot |f(x-2^{-k}y)|^{r} (1+|y|)^{(\frac{1}{r}-1)d}dy\cdot [ M(|f|^r)(x)]^{\frac{1-r}{r}}\\
\les& 2^{(j-k)(\frac{1}{r}-1)d}\int_{\R^d}  |\psi_k(y)|\cdot |f(x-y)|^{r} (1+|2^k y|)^{(\frac{1}{r}-1)d}dy\cdot [ M(|f|^r)(x)]^{\frac{1-r}{r}}\\
\les& 2^{(j-k)(\frac{1}{r}-1)d}  [ M(|f|^r)(x)]^{\frac{1}{r}}.
}
Thus we complete the proof.
\end{proof}

\begin{remark}

\eqref{convomaxi} is crucial for us deriving the nonlinear etimates in the next subsection. By \eqref{convomaxi} we obtain for $r<1$
\EQ{\label{eq:hhimprove}
|\psi_k*(\dot P_j f \cdot \dot P_j g)|\les 2^{(j-k)d(\frac{1}{r}-1)} \big[M(|\dot P_j f\cdot \dot P_j g|^r)(x)\big]^{\frac{1}{r}}.
}
In particular, the above estimate is useful to handle the high-high to low frequency interactions. 
\end{remark}

A key ingredient in our proof is the boundedness of the Schr\"odinger propagator on Hardy space $\Hl^1$.  This was proved by Miyachi \cite{Miyachi}. 

\begin{lem}[Hardy space boundness]\label{lem:Sch}
We have
\EQ{
\|e^{it\Delta}f\|_{F^0_{1,2}(\R^d)}\les& (1+|t|)^{d/2}\|f\|_{F_{1,2}^{d}(\R^d)},\\
\|e^{it\Delta}P_{\leq 0}f\|_{\dot F^0_{1,2}(\R^d)}\les& (1+|t|)^{d/2}\|f\|_{\dot F^0_{1,2}(\R^d)}.
}
\end{lem}

The second ingredient is the nonlinear estimates in Hardy space.   First we derive a Leibniz rule in the Hardy space.

\begin{lem}\label{lem:Leib}
Let $s>0$, $p_i, q_i \in (1,\infty)$, $\frac{1}{p_i}+\frac{1}{q_i}=1$, $i=1,2$.  Then
\begin{align}
\|D^s(fg)\|_{\Hl^1}\les \norm{D^sf}_{L^{p_1}}\norm{g}_{L^{q_1}}+\norm{f}_{L^{p_2}}\norm{D^sg}_{L^{q_2}}.
\end{align}
\end{lem}
\begin{proof}
By Bony's paraproduct decomposition, we have
\EQ{
fg=&\sum_{|j-m|\leq 3}\dot P_m(\dot P_j f\cdot \dot P_{\leq m-5} g)+\sum_{|j-m|\leq 3}\dot P_m(\dot P_{\leq m-5} f\cdot \dot P_{j} g)+\sum_{j\geq m-3}\dot P_m(\dot P_{j} f\cdot \dot P_{j} g)\\
:=&I+II+III.
}
For the term $I$, by Lemma \ref{lem:keyesti} we have
\EQ{
\norm{D^s(I)}_{\Hl^1}\les& \normb{\norm{2^{ms}\dot P_m(\dot P_m f\cdot \dot P_{\leq m-5} g)}_{l_m^2}}_{L^1}\\
\les& \normb{\norm{M(|2^{ms} \dot P_m f\cdot \dot P_{\leq m-5} g|^{1/2})^2}_{l_m^2}}_{L^1}\\
\les& \normb{\norm{2^{ms} \dot P_m f}_{l_m^2}M(g)}_{L^1}\les \norm{D^s f}_{L^{p_i}}\norm{g}_{L^{q_i}}.
}
The term $II$ is similar as the term $I$ and we can get
\EQ{
\norm{D^s(II)}_{\Hl^1}
\les&  \norm{f}_{L^{p_i}}\norm{D^sg}_{L^{q_i}}.
}
For the term $III$, applying \eqref{eq:hhimprove} by taking $r<1$ such that $ d(1/r-1)<s$, we get
\EQ{
\norm{D^s(III)}_{\Hl^1}\les& \normb{\norm{\sum_{j\geq m-3}  2^{(m-j)s}\dot P_m(2^{js}\dot P_j f\cdot \dot P_j g)}_{l_{m}^2}}_{L^1}\\
\les& \normb{\norm{\sum_{j\geq m-3}  2^{(m-j)s} 2^{(j-m)d(1/r-1)} [M(|2^{js}\dot P_j f\cdot \dot P_j g|^{r})]^{1/r}}_{l_{m}^2}}_{L^1}\\
\les& \normb{M(|2^{js}\dot P_j f\cdot \dot P_j g|^{r})}^{1/r}_{L^{1/r}l_j^{2/r}}.
}
Then by  Lemma \ref{lem:maxi} and the H\"older inequality we have
\EQ{
\norm{D^s(III)}_{\Hl^1}\les&  \normb{2^{js}\dot P_j f\cdot M g}_{L^1l_j^2}\\
\les&\norm{D^s f}_{L^{p_i}}\norm{g}_{L^{q_i}}.
}
Therefore, we complete the proof.
\end{proof}

\begin{remark}
It seems to us Lemma \ref{lem:Leib} is new\footnote{By private communication, the authors learned that some generalized Leibniz rules in Hardy space were obtained independently in \cite{LX}}.  The fractional Leibniz rule has been extensively studied, see \cite{DLi} for a comprehensive survey of the current results.  The classical inequality reads
\EQ{\label{eq:Leibniz2}
\norm{D^s(fg)}_{r}\les \norm{D^s f}_{p_1}\norm{g}_{q_1}+\norm{f}_{p_2}\norm{D^s g}_{q_2}
}
where $s>0$, $1<p_i,q_i\leq \infty$ with $1/r=1/p_i+1/q_i$: $i=1,2$, and $1/(1+s)<r\leq \infty$.

The same argument of the proof of Lemma \ref{lem:Leib} can be applied to the Hardy space boundedness of a class of bilinear operators.  This  is known as compensated compactness.   See \cite{CLMS}.  Consider the bilinear operator of the form
\EQ{
B_m(f,g)(x)=\int_{\R^d\times \R^d} m(\xi_1,\xi_2)\hat f(\xi_1)\hat g(\xi_2) e^{ix(\xi_1+\xi_2)}d\xi_1 d\xi_2.
}
We decompose 
\EQ{m=m_{HL}+m_{LH}+m_{HH},} 
where 
\EQ{
m_{HL}(\xi,\eta)=&\sum_{j,k:j\geq k+5}\chi_j(\xi)\chi_k(\eta)\\
m_{LH}(\xi,\eta)=&\sum_{j,k:k\geq j+5}\chi_j(\xi)\chi_k(\eta)\\
m_{HH}(\xi,\eta)=&\sum_{j,k:|j-k|\leq 4}\chi_j(\xi)\chi_k(\eta).
}
By the similar argument as proving Lemma \ref{lem:Leib}, we can obtain the following proposition, although we will not need it in this note.

\begin{prop}
Assume $B_{\sigma}(f,g): L^2(\R^d)\times L^2(\R^d)\to L^1(\R^d)$ is a bounded bilinear operator, for $\sigma\in \{m_{HL}, m_{LH}, m_{HH}\}$, and some cancellation property: there exists $\gamma>0$ such that 
\EQ{
\norm{\dot P_k B_m(\dot P_j f, \dot P_l g)}_{L^1(\R^d)}\les 2^{(k-j)\gamma}\norm{\dot P_j f}_{L^2}  \norm{\dot P_l g}_{L^2}
}
for all $k, j, l\in \Z$ with $2^k\ll 2^j\sim 2^l$.  Then
$\norm{B_m(f, g)}_{\Hl^1(\R^d)}\les  \norm{f}_{L^2}  \norm{g}_{L^2}$.
\end{prop}
\end{remark}

Besides the fractional Leibniz rule, the fractional chain rule is also important in applications, especially in dealing with the non-algebraic nonlinearity $F(u)=|u|^{p-1}u$ when $p$ is not odd. The classical fractional chain rule says (see Tao \cite{Taonote}): if $s\in (0,1)$, $1<m,t,q<\infty$ with $\frac{1}{m}=\frac{p-1}{t}+\frac{1}{q}$
\EQ{
\|D^sF(u)\|_{L^m(\R^d)}\lesssim \norm{u}_{L^t}^{p-1}\norm{D^s u}_{L^q}.
}
In the lemma below, we extend the fractional chain rule to the Hardy space when $m=1$, although we do not need it in this note. 

\begin{lem}[Fractional chain rule in Hardy space]\label{lem:fracChain}
Let $F(u)$ be a power-type nonlinearity with exponent $p\geq 1$, namely of the form $|u|^p$. Assume $s\in (0,1)$, $1<t,q<\infty$ with $1=\frac{p-1}{t}+\frac{1}{q}$. Then
\begin{align}
\|D^sF(u)\|_{F^0_{1,2}(\R^d)}\lesssim \norm{u}_{L^t}^{p-1}\norm{D^s u}_{L^q}.
\end{align}
\end{lem}
\begin{proof}
We assume first the pointwise estimate: for $r\in (0,1)$
\EQ{\label{eq:pointwise}
|\dot P_j [F(u)](x)|\les \sum_{k}\min(2^k,1)\bigg[&M(|Mu|^{p-1})(x)M(\dot P_{j+k}u)(x)\\
&+ 2^{kd(\frac{1}{r}-1)}\brk{M\bbrk{\aabs{M(|Mu|^{p-1})M(\dot P_{j+k}u)}^r}(x)}^{1/r}\bigg].
}
Then
\EQ{
&\|D^sF(u)\|_{F^0_{1,2}(\R^d)}\\
\sim &\norm{2^{sj}\dot P_j F(u)}_{L^1l_j^2}\\
\les & \norm{\sum_{k}\min(2^k,1)2^{-ks}2^{(j+k)s}M(|Mu|^{p-1})\cdot M(\dot P_{j+k}u)}_{L^1l_j^2}\\
&+ \normo{\sum_{k}\min(2^k,1)2^{-ks}2^{(j+k)s} 2^{kd(\frac{1}{r}-1)}\brk{M\bbrk{\aabs{M(|Mu|^{p-1})M(\dot P_{j+k}u)}^r}(x)}^{1/r}\bigg]}_{L^1l_j^2}\\
:=&A+B.
}
For the term $A$, since $s>0$ we get by Lemma \ref{lem:maxi} and the H\"older inequality that
\EQ{
A\les& \norm{M(|Mu|^{p-1})\cdot \norm{M(2^{js} \dot P_{j}u)}_{l_j^2}}_{L^1}\\
\les& \norm{u}_{L^t}^{p-1}\norm{D^s u}_{L^q}.
}
For the term $B$, since $s\in (0,1)$, we choose $r<1$ but sufficiently close to 1 such that $s-d(\frac{1}{r}-1)>0$. Then we get by Lemma \ref{lem:maxi} and the H\"older inequality that
\EQ{
B\les& \normo{\brk{M\bbrk{\aabs{M(|Mu|^{p-1})M(2^{js}\dot P_{j}u)}^r}^{1/r}}}_{L^1l_j^2}\\
\les& \norm{u}_{L^t}^{p-1}\norm{D^s u}_{L^q}.
}

It remains to prove \eqref{eq:pointwise}, we may assume $j=0$ by scaling and $x=0$ by translation. Then by fundamental theorem of calculus we have
\EQ{
F(u)=F(P_{\leq 0}u)+\int_0^1 F'(P_{\leq 0} u+tP_{>0}u) dt\cdot P_{>0}u
}
and thus
\EQ{
\dot P_0F(u)(0)=&\dot P_0[F(P_{\leq 0}u)](0)+\dot P_0[\int_0^1 F'(P_{\leq 0} u+tP_{>0}u) dt\cdot P_{>0}u](0)\\
:=&I+II.
}
For the term $II$, we have
\EQ{
II=\sum_{k>0}\sum_{|j-k|\leq 5}\dot P_0\brk{\dot P_j[\int_0^1 F'(P_{\leq 0} u+tP_{>0}u) dt]\cdot \dot P_{k}u}.
}
By \eqref{eq:hhimprove} we have
\EQ{
|II|\les& \sum_{k>0}\sum_{|j-k|\leq 5}2^{kd(\frac{1}{r}-1)}\brk{M\bbrk{\aabs{\dot P_j[\int_0^1 F'(P_{\leq 0} u+tP_{>0}u) dt]\cdot \dot P_{k}u}^r}(0)}^{1/r}\\
\les& \sum_{k>0}2^{kd(\frac{1}{r}-1)}\brk{M\bbrk{\aabs{M(|Mu|^{p-1}) \dot P_{k}u}^r}(0)}^{1/r}.
}
For the term $I$, we have
\EQ{
I=&\dot P_0[F(P_{\leq 0}u)-F(P_{\leq 0}u)(0)](0)\\
=&\dot P_0\brk{\int_0^1 F'(P_{\leq 0} u(0)+t[P_{\leq 0}u-P_{\leq 0}u(0)]) dt\cdot [P_{\leq 0}u-P_{\leq 0}u(0)]}\\
=&\sum_{k\leq 0} \dot P_0\brk{[\int_0^1 F'(P_{\leq 0} u(0)+t[P_{\leq 0}u-P_{\leq 0}u(0)]) dt]\cdot [\dot P_{k}u-\dot P_{k}u(0)]}\\
=&\sum_{k\leq 0} \int \brk{[\int_0^1 F'(P_{\leq 0} u(0)+t[P_{\leq 0}u-P_{\leq 0}u(0)]) dt]\cdot \int_0^1 \nabla \dot P_{k}u(sy)\cdot yds}\chi(y)dy.
}
By Lemma \ref{lem: scalarmaxiesti} we have
\EQ{
|\nabla \dot P_{k}u(sy)|\les 2^k M(|\dot P_k u|^r)(0)^{1/r}(1+|y|)^{d/r}.
}
Therefore we get
\EQ{
I\les&\sum_{k\leq 0} 2^k M(|Mu|^{p-1})(0)M(|\dot P_k u|^r)(0)^{1/r}.
}
The proof is completed.
\end{proof}

\begin{lem}[Nonlinear estimate in Hardy space]\label{lem:nonlinear}
\begin{align}
\||u|^4 u\|_{F^3_{1,2}(\R^3)}\lesssim \|u\|_{\infty}^{3}\|J^3 u\|_2 \norm{u}_2.
\end{align}
\end{lem}
\begin{proof}
For the low frequency component, it's easy to see
\EQ{
\|P_{\leq 0}(|u|^4 u)\|_{L^1}\les \|u\|_{\infty}^{3}\|u\|_2 \norm{u}_2.
}
Thus it remains to prove 
\EQ{
\|P_{\geq 1}(|u|^4 u)\|_{\dot F^3_{1,2}}\lesssim \|u\|_{\infty}^{3}\|J^3 u\|_2 \norm{u}_2.
}
The above inequality follows from Lemma \ref{lem:Leib} and \eqref{eq:Leibniz2}.
\end{proof}

\subsection{Proof of Theorem \ref{Th1}.}

From the duhamel formula, we have 
\begin{equation}
u= e^{it\triangle}u_0 -i\mu\int_0^t e^{i(t-s)\triangle}|u(s)|^{4}u(s)ds.
\end{equation}
Let $T$ be sufficiently large. For simplicity, we define $\|u\|_{X_T}:= \|t^{3/2}u\|_{L^{\infty}_{x, t\in [T, \infty)}}$. Then applying the decay estimate for the Schr\"odinger equation \eqref{decay-NLS}, we obtain that
\EQ{
\|u\|_{X_T} &\leq \|u_0\|_{L^1_x}+\int_0^{\infty}\|e^{-is\triangle}|u(s)|^{4}u(s)\|_{L^1_x}ds\\
& \leq \|u_0\|_{L^1_x}+\int_0^{T}\|e^{-is\triangle}|u(s)|^{4}u(s)\|_{L^1_x}ds+\int_T^{\infty}\|e^{-is\triangle}|u(s)|^{4}u(s)\|_{L^1_x}ds\\
&:=\|u_0\|_{L^1_x}+I+II.
}

For the term $I$, using Strichartz estimates we can easily get
\EQ{
I\leq C_T.
}
It remains to estimate the term $II$.  By Lemma \ref{lem:Sch} we have
\EQ{
II\les& \int_T^{\infty}\|e^{-is\triangle}|u(s)|^{4}u(s)\|_{F^0_{1,2}}ds\\
\les& \int_T^{\infty}s^{3/2}\||u(s)|^{4}u(s)\|_{F^3_{1,2}}ds\\
\les& \int_T^{\infty}s^{3/2}\|D^3u(s)\|_{L^2}\norm{u(s)}_{L^2}\norm{u(s)}_{L^\infty}^{3}ds\\
\les& \norm{u}^2_{L_{t\in [T,\infty)}^2L_x^\infty}\norm{u}_{X_T}
}
where we used $\norm{u}_{L_t^\infty H^3}\leq C$.
Thus we obtain 
\EQ{
\|u\|_{X_T}\les \norm{u_0}_{L_x^1}+C_T+\norm{u}^2_{L_{t\in [T,\infty)}^2L_x^\infty}\norm{u}_{X_T}.
}
Since $\norm{u}_{L_{t\in \R}^2L_x^\infty}<\infty$ by \eqref{eq:spacetimebound} and Sobolev embedding, then $\norm{u}_{L_{t\in [T,\infty)}^2L_x^\infty}\to 0$ as $T\to \infty$.
Taking $T>0$ sufficiently large, we get
$\norm{u}_{X_T}\les 1$.

Now we show that $\lim_{t\to +\infty}e^{-it\Delta}u(t)=\phi_+$ in $L^1$.  It suffices to show $\int_0^\infty e^{-is\Delta} |u|^4u ds$ is convergent in $L^1$, which can be obtained by the previous argument. 
Moreover, we have
\EQ{
\norm{u(t)-e^{it\Delta}\phi_+}_{L^\infty}\les& \normo{\int_t^\infty e^{i(t-s)\triangle}|u(s)|^{4}u(s)ds}_{L^\infty}\\
\les& \normo{\int_t^{2t} e^{i(t-s)\triangle}|u(s)|^{4}u(s)ds}_{L^\infty}+\normo{\int_{2t}^\infty e^{i(t-s)\triangle}|u(s)|^{4}u(s)ds}_{L^\infty}\\
\les& \int_t^{2t} \normo{e^{i(t-s)\triangle}|u(s)|^{4}u(s)}_{H^2}ds+\int_{2t}^\infty |t-s|^{-3/2}\norm{|u(s)|^{4}u(s)}_{L^1}ds\\
\les& \int_t^{2t} \norm{u(s)}_{L^\infty}^4\normo{u(s)}_{H^2}ds+\int_{2t}^\infty |s|^{-3/2}\norm{u(s)}_{L^\infty}^{3}\norm{u(s)}_{L^2}^2ds\les t^{-5}.
}

Conversely, given final data $\phi_+\in H^3\cap L^1$, by the classical results one can get a global solution $u\in C(\R:H^3)\cap L_{t,x}^{10}$ such that $e^{-it\Delta}u(t)\to \phi_+$ in $H^3$ as $t\to \infty$.  Moreover, 
\begin{equation}
u= e^{it\triangle}\phi_+ -i\mu\int_{t}^\infty e^{i(t-s)\triangle}|u(s)|^{4}u(s)ds.
\end{equation}
Repeating the previous arguments, we see $u\in X_T$ and $e^{-it\Delta}u(t)\to \phi_+$ in $L^1$ as $t\to \infty$. Thus the proof is completed.

\subsection*{Acknowledgement}
The authors are grateful to Professor Kenji Nakanishi for the precious discussions.

\end{document}